\documentclass[10pt,letterpaper]{amsart}
\usepackage[utf8]{inputenc}
\usepackage{color}
\usepackage{cite}
\usepackage{hyperref}
\usepackage{amsmath}
\usepackage{amsthm}
\usepackage{amsfonts}
\usepackage{amssymb}
\usepackage[english]{babel}
\usepackage{lmodern}
\usepackage{graphicx}

\usepackage{geometry}
\hypersetup{linktocpage,colorlinks, citecolor={magenta}}
\numberwithin{equation}{section}

\newtheorem{theorem}{Theorem}
\newtheorem{prop}{Proposition}[section]
\newtheorem{lemma}[prop]{Lemma}

\newtheorem{claim}[prop]{Claim}

\theoremstyle{definition}
\newtheorem{remark}[prop]{Remark}
\newtheorem{definition}[prop]{Definition}

\newcommand{\R}{\mathbb{R}}
\renewcommand{\div}{\mathrm{div}}
\renewcommand{\det}{\mathrm{det}}

\newcommand{\muz}{\mathbf{\mu}_{0}}

\newcommand{\id}{\mathrm{Id}}

\renewcommand{\phi}{\varphi}
\renewcommand{\epsilon}{\varepsilon}

\newcommand{\D}{\mathrm{D}}

\newcommand{\1}{\mathsf{1}}
\newcommand{\2}{\mathsf{2}}

\newcommand{\hal}{\frac{1}{2}}

\newcommand{\XN}{\mathrm{X}_N}
\newcommand{\YN}{\mathrm{Y}_N}
\newcommand{\FT}{\mathcal{F}_N}
\newcommand{\FTi}{\mathcal{F}_{\infty}}

\newcommand{\Esp}{\mathbb{E}}
\renewcommand{\P}{\mathbb{P}}
\newcommand{\Teps}{\mathrm{T}_{\epsilon}}
\newcommand{\vTeps}{\vec{\mathrm{T}}_{\epsilon}}
\newcommand{\F}{\mathsf{F}}

\newcommand{\kk}{\mathsf{k}}

\newcommand{\Ani}{\mathsf{Ani}}
\newcommand{\Ele}{\mathsf{Ele}}
\newcommand{\Points}{\mathsf{Points}}

\newcommand{\rr}{\mathsf{r}}
\newcommand{\vr}{\vec{\rr}}

\newcommand{\muep}{\mu_{\epsilon}}
\newcommand{\I}{\mathrm{I}}
\newcommand{\II}{\mathrm{II}}
\newcommand{\III}{\mathrm{III}}

\newcommand{\Ocross}{O_{\star}}
\newcommand{\llN}{\ell_N}

\newcommand{\LN}{\mathbf{M}_N}
\newcommand{\bXN}{\mathbf{X}_N}

\newcommand{\EnergiePoints}{\mathsf{EnerPts}} 

\renewcommand{\d}{\mathrm{d}}
\newcommand{\dx}{\d x}
\newcommand{\diagc}{\left(\R^2 \times \R^2\right) \setminus \triangle}

\newcommand{\supp}{\mathrm{supp }}

\newcommand{\El}{\mathrm{E}}
\renewcommand{\r}{\mathsf{r}}

\newcommand{\veta}{\vec{\eta}}

\newcommand{\G}{\mathcal{G}}

\newcommand{\EYz}{\El^{(0, \mathrm{Y})}}

\newcommand{\EYe}{\El^{(\epsilon, \mathrm{Y})}}

\newcommand{\AniYz}{\Ani^{(0)}}
\newcommand{\AniYe}{\Ani^{(\epsilon)}}

\newcommand{\bYN}{\mathbf{Y}_N}

\newcommand{\0}{\mathsf{0}}

\newcommand{\A}{\mathsf{A}}

\newcommand{\vzeta}{\vec{\zeta}}
\newcommand{\PNbeta}{\mathbb{P}_{N}^{\beta}}
\newcommand{\QNbeta}{\mathbb{Q}_{N}^{\beta}}
\newcommand{\ZNbeta}{\mathrm{Z}_{N}^{\beta}}
\newcommand{\Lploc}{\mathrm{L}^p_{\mathrm{loc}}}

\newcommand{\mmuep}{\mathfrak{m}_{\epsilon}}
\newcommand{\mmuz}{\mathfrak{m}_{\mathsf{0}}}
\newcommand{\lN}{\ell_N}
\newcommand{\bphi}{\overline{\varphi}}
\newcommand{\epN}{\epsilon_N}
\newcommand{\cmin}{c_{\mathrm{min}}}
\newcommand{\llangle}{\left \langle}
\newcommand{\rrangle}{\right \rangle}

\renewcommand{\A}{\mathsf{A}_{\mathrm{1}}}

\newcommand{\dN}{\delta_N}

\newcommand{\Rd}{\mathrm{R}_{\delta}}
\newcommand{\tX}{\widetilde{X}_{\delta}}

\newcommand{\mR}{\mathrm{R}}
\newcommand{\nhmu}{\nabla h^{\mu}}
\newcommand{\nhmue}{\nabla h^{\mu}_{\veta}}
\newcommand{\EYze}{\EYz_{\veta}}
\newcommand{\EYee}{\EYe_{\veta}}
\newcommand{\fetai}{\mathsf{f}_{\eta_i}}
\newcommand{\ff}{\mathsf{f}}
\newcommand{\Ll}{\mathrm{L}}
\renewcommand{\aa}{\mathsf{a}}
\def\corO{}

\begin{document}
\title{A local CLT for linear statistics of 2D Coulomb gases}
\author{Thomas Lebl\'{e} and Ofer Zeitouni}
\address{Thomas Lebl\'{e}: Universit\'{e} 
de Paris, CNRS, MAP5 UMR 8145, F-75006 Paris, France.}
\email{thomasleble@gmail.com}
\address{Ofer Zeitouni\\
Department of Mathematics\\
Weizmann institute of Science, Rehovot 76100, Israel.}
\email{ofer.zeitouni@weizmann.ac.il}
\thanks{This project has received funding from the European Research Council (ERC) under the European Union's Horizon 2020 research and innovation programme (grant agreement No. 692452).  T.L. acknowledges the support of the Institute for Advanced Study and the Florence Gould foundation. \\
We thank Sylvia Serfaty for providing us with early drafts of \cite{SerfatyCLT}, and for making some of the statements therein easily quotable for our purposes.
}
\date{\today}
\begin{abstract}
We prove a local central limit theorem for fluctuations of linear statistics of smooth enough test functions under the canonical Gibbs measure of two-dimensional Coulomb gases at any positive temperature. The proof relies on the \textit{global} central limit theorem of \cite{BBNY, LebSerCLT} and a new decay estimate for the characteristic function of such fluctuations.
\end{abstract}
\maketitle
 \section{Introduction}
\subsection{Goal of the paper}
Let $\PNbeta$ be the canonical Gibbs measure (defined below in \eqref{def:Gibbs}) of the two-dimensional Coulomb gas with $N$ particles $\XN = (x_1, \dots, x_N)$ at inverse temperature $\beta > 0$, let $\muz$ be the associated \textit{equilibrium measure}  and let $\bXN := \sum_{i=1}^N \delta_{x_i}$ be the (random) point configuration associated to the $N$ (random) positions of the particles. 
We introduce the random signed measure 
\begin{equation}
  \label{eq-MN}
    \LN := \bXN - N \muz
\end{equation}
and, when $\phi$ is a continuous real-valued test function on $\R^2$, we define the \emph{fluctuations of $\phi$} as:
\begin{equation}
\label{def:fluctuations} \langle \phi, \LN \rangle := \int_{\R^2} \phi(x) \d\LN(x) = \sum_{i=1}^N \phi(x_i) - N \int_{\R^2} \phi(x) \d\muz(x). 
\end{equation}
The papers \cite{LebSerCLT, BBNY} prove a central limit theorem (CLT) for \textit{fluctuations of linear statistics}, namely quantities of the type $ \langle \phi, \LN \rangle$ where $\varphi$ is smooth enough, possibly living at some mesoscopic scale, and supported within the support of $\muz$. The random variable $\langle \phi, \LN \rangle$ converges in distribution  as $N \to \infty$ to a certain Gaussian random variable $Z$ (with explicit mean and variance). The goal of the present paper is to prove a \textit{local} central limit theorem, namely that:
$$
\text{for any $\aa \in (-\infty, +\infty)$, }  \PNbeta\left[  \langle \phi, \LN \rangle \in (\aa - \epsilon_N, \aa + \epsilon_N) \right] \sim_{N \to \infty}  \P\left[  Z \in (\aa - \epsilon_N, \aa + \epsilon_N) \right],
$$
where the characteristic size $\epsilon_N$ of the interval is allowed to go to $0$ at certain speeds that we specify below, depending on the scale at which the test function lives.  

The proof relies on the existing CLT and a crucial new estimate of the decay at infinity of the fluctuations' characteristic function. 
We use the transportation technique of \cite{LebSerCLT, BLS}, 
which is closely related to the loop equations toolbox as in e.g. 
\cite{BG1, BBNY}, but instead of using it to study the Laplace transform of a given fluctuation as in the papers cited, here we turn to study its Fourier transform. We plan to use the  new decay estimate we obtained
 in further investigations of 2D Coulomb gases.

\subsection{Setting}
\label{subsec-settings}
\begin{itemize}
\item The \textit{equilibrium measure} $\muz$ will be taken as a probability measure on $\R^2$ that satisfies: $\muz$ has a compact support $\Sigma$ and possesses a density $\mmuz$ with respect to the Lebesgue measure that is bounded below by a positive constant $\cmin$ on $\Sigma$. We assume that $\Sigma$ is the closure of its interior, that
the boundary $\partial \Sigma$ is a piece-wise $C^1$ curve, and that the density $\mmuz$ is of class $C^{3}$ on $\Sigma$. 

\item The \textit{confining potential} $\zeta$ will simply denote a non-negative continuous function that vanishes on the support $\Sigma$ of $\muz$, is positive outside $\Sigma$, and grows at infinity in a strongly confining
  way, i.e. $\liminf_{|x| \to \infty} \frac{\zeta(x)}{\log|x|} > 2$. We will denote by $\vzeta(\XN)$ the quantity 
$\vzeta(\XN) := 2N \sum_{i=1}^N \zeta(x_i)$.
\end{itemize}
 
\begin{definition}[The two-dimensional Coulomb energy]
\label{def:energy}
To any state of the system $\XN = (x_1, \dots, x_N)$ we associate its \textit{logarithmic energy} $\F(\XN, \muz)$ given by
\begin{equation}
\label{def:FXN}
\F(\XN, \muz) := \hal \iint_{\diagc} - \log|x-y| (\d\bXN - \d\muz)(x) (\d\bXN - \d\muz)(y),
\end{equation}
where $\triangle$ denotes the diagonal in $\R^2 \times \R^2$.  Recalling that $-\log$ is (up to a multiplicative constant) the interaction potential for two-dimensional electrostatics, we can think of $\F(\XN, \muz)$ as being the electrostatic energy of a system consisting of $N$ negative point charges placed at $(x_1, \dots, x_N)$ and a charged background of opposite sign whose density is given by $\muz$.
\end{definition}

\begin{definition}[Canonical Gibbs measure of the two-dimensional Coulomb gas]
\label{def:Gibbs}
For all $\beta > 0$ we consider the probability measure $\PNbeta$ on $\left(\R^2\right)^N$, with density
\begin{equation}
\label{eq-gibbs}
\d\PNbeta(\XN) := \frac{1}{\ZNbeta} \exp\left( - \beta \left( \F(\XN, \muz) + \vzeta(\XN) \right) \right) \d\XN,
\end{equation}
where $\ZNbeta$ is the normalizing constant, or partition function, given by the integral
\begin{equation*}
\ZNbeta := \int_{\left(\R^2\right)^N} \exp\left( - \beta \left( \F(\XN, \muz) + \vzeta(\XN) \right) \right) \d\XN.
\end{equation*}
The measure $\PNbeta$ is called the \textit{canonical Gibbs measure}, and $\beta$ the \textit{inverse temperature parameter}.
\end{definition}

\textit{Henceforth, we fix $N \geq 2$ and an arbitrary value of $\beta > 0$. We let $\XN$ be a random variable in $(\R^2)^N$ distributed according to the Gibbs measure $\PNbeta$, and $\bXN$ be the associated random atomic measure of mass $N$ on $\R^2$.}

\begin{remark}
Often, one defines a Coulomb Gibbs measure by
choosing 
an ``external potential" $V$ that has a certain regularity and growth at infinity. 
That is, given such function $V$, one defines the measure
\begin{equation}
\label{eq-gibbs-alt}
\d\QNbeta(\XN) := \frac{1}{\widehat
\ZNbeta} \exp\left( - \beta \left( \F(\XN, 0) + V(\XN) \right) \right) \d\XN,
\end{equation}
where $\widehat \ZNbeta$ is a normalization constant and $V(\XN)=N \sum_{i=1}^N V(x_i)$.
Classical potential theory then yields, under the above mentioned mild assumptions on $V$, 
the existence of an equilibrium measure $\muz$ so that, on its support,
$\int \log|x-y|d\muz(y)=V(x)+C_{\muz}$, for an appropriate constant $C_{\muz}$. One then obtains 
\eqref{eq-gibbs} with the confining potential $\zeta$ given, off the support of $\muz$, by the
difference of $V$ and  the
 logarithmic potential of $\muz$.
Here we prefer to forget $V$ in order to work directly with the measure itself. Since we are restricting ourselves to test functions supported inside the support $\Sigma$ of $\muz$, the confining potential $\zeta$ will play almost no role. We note that $\muz$ is still the 
``equilibrium measure" of the system in the sense that  as $N \to \infty$,
 the empirical measure $L_N$ converges weakly to $\muz$ almost surely. 
\end{remark}

\subsection{Statements}
\label{sec:statements}
We fix a compactly supported function $\bphi$ of class $C^5$ on $\R^2$ and we take either:
\begin{itemize}
\item (Macroscopic case) $\varphi_N = \bphi$, assuming $\bphi$ supported in the interior of $\Sigma$. In this case we set $\lN := 1$.
\item (Mesoscopic case) $\varphi_N := \bphi\left( \frac{\cdot - \bar{x}_N}{\lN} \right)$, where $\lN$ is such that $N^{-1/2} \ll \lN \ll 1$ and $\{\bar{x}_N\}_N$ is a sequence of points in $\Sigma$ that stays at a certain positive distance from $\partial \Sigma$.
\end{itemize}
We will usually omit the subscript $N$ and simply write $\varphi$, with an implicit dependence on $N$ in the mesoscopic case. 
\begin{remark}
The test functions that we consider always live in the interior of the support of the equilibrium measure, in fact in both the macroscopic and mesoscopic cases their supports stay at some positive distance from the boundary. Allowing the test functions to have a support intersecting (or even close to) the boundary would raise several technical problems: 
\begin{itemize}
\item There is the well-known fact, see e.g. \cite{Pasturcut} in the one dimensional case,
 that a CLT does not hold in general if the support of $\muz$ has more than one connected component (one needs to assume some compatibility conditions on the test function).
\item The effect of such a ``boundary" test function on the equilibrium measure and its support is very subtle, see \cite{serfaty2018quantitative} for an analysis
\item The local laws that control the energy at small scales and that
are needed to treat mesoscopic cases have not been proven close to the boundary.
\end{itemize}
\end{remark}

The main result of the present paper is the following theorem. In it and the rest of the paper,
we use the standard notation $a_N\ll b_N$ to mean that $a_N/b_N\to_{N\to\infty} 0$.
\begin{theorem}[Local CLT for fluctuations]
\label{theo}
Let $Z$ be a Gaussian random variable with mean $b_Z$ equal to
$0$ (in the mesoscopic case) or $ \frac{1}{2\pi} \left( \frac{1}{\beta} - \frac{1}{4} \right) \int_{\R^2} \Delta \varphi(x) \log \mmuz(x) \d x $ (in the macroscopic case), and variance $\sigma_Z^2=\frac{1}{2\pi \beta} \int_{\R^2} |\nabla \phi|^2$.  Let $\{\epN\}_N$ be a sequence that satisfies\footnote{The optimal lower bound probably does not have the $\log N$ factor.}:
\begin{equation}
\label{epNlN}
\frac{\lN^{-2}}{N} \log N \ll \epN \ll 1.
\end{equation}

Then the following \textit{local CLT} 
holds\footnote{We observe that since $Z$ has a Gaussian density the right hand-side of \eqref{localCLT} is bounded above and below by positive constants (depending on $\aa$) times $\epsilon_N$.}: for all real number $\aa$, we have:
\begin{equation}
\label{localCLT}
\PNbeta \left[  \langle \phi, \LN \rangle \in (\aa - \epsilon_N, \aa + \epsilon_N) \right] = \P\left[  Z \in (\aa - \epsilon_N, \aa + \epsilon_N) \right]\left(1 + o_N(1)\right),
\end{equation}
and the convergence is uniform for $\aa$ in
compacts.
\end{theorem}

The new technical ingredient in the proof of Theorem \ref{theo}
is a decay estimate on the Fourier transform of fluctuations. The (sequence of) 
test functions $\phi$ having been chosen as above, we introduce the characteristic function $\FT$:
\begin{equation}
\label{eq:FTvarphi}
\FT : \omega \mapsto \FT(\omega) := \Esp \left[ e^{i \omega \langle \phi, \LN \rangle}  \right].
\end{equation}
We write $C(\bphi, \muz, \beta)$ to denote a positive constant that depends only on $\bphi, \muz$ and the value of $\beta$.
\begin{prop}[Decay estimate on the Fourier transform of fluctuations]
\label{prop:FTdecayA}
We have that
\begin{equation}
\label{FTdecayA} |\FT(\omega)| \leq C(\bphi, \muz, \beta)  \left( \frac{1}{\omega^2} + \frac{\lN^{-2}}{N} \right),\quad \mbox{\rm for $|\omega| > 1$.}
\end{equation}
\end{prop}

 \begin{remark}
   In the related, one-dimensional context of linear statistics for
   unitary
 random matrices, a CLT with fast convergence in total variation norm (that
 in particular implies a 
 local CLT) was proven in \cite{johansson1997random} (see also the
 recent \cite{johansson2020multivariate} for multidimensional extensions).
 In the Hermitian random matrix context,
 \cite{bourgade2016fixed}  study the Fourier transform of fluctuations in order to get
 better estimates on the speed of convergence
 of linear statistics to the limiting Gaussian random variable.
\end{remark}

In Section \ref{sec:proofwithdecay}, we present a proof of Theorem \ref{theo},
assuming the decay estimate \eqref{FTdecayA}. We prove the latter
in Section~\ref{sec:proofdecay}. In the Appendix, we gather existing results around the logarithmic interaction energy, its electric formulation, its behavior under a change of variables, and some useful \textit{a priori} bounds.

\subsection{Notation.} 
\label{subsec-notation}
Throughout the paper, constants 
  $K_i$ are universal, i.e. independent of $\omega$ or $N$ (however, they
may depend on $\muz$ and $\beta$.)
  We use throughout the standard $o(\cdot)$ and $O(\cdot)$
notation. Thus, $b=o_n(a)$ 
means that $b/a\to_{n\to\infty} 0$ and $b=O_n(a)$ means that $b/a$ remains bounded by a universal constant
as $n$ varies. When the parameter $n$ is clear from the context, we write
$a_n=o(b_n)$, etc. The notation $a_N\ll b_N$ is thus equivalent 
to the implicit notation $a=o_N(b)$. Finally, the notation $\Ocross(t)$
means a quantity whose absolute value
is bounded above by $Ct$ where $C$ may depend on other parameters, e.g.
$N,\omega$, but not on $t$.

When no confusion can occur,
  we use the symbols $\P$, $\Esp$ to denote probabilities and expectations
with respect to the measure determined by the random variables involved. Thus, 
the left hand side of \eqref{localCLT} equals
$\P \left[  \langle \phi, \LN \rangle \in (\aa - \epsilon_N, \aa + \epsilon_N) \right]$.

\section{Proof of Theorem \ref{theo} using Proposition \ref{prop:FTdecayA}}
\label{sec:proofwithdecay}
In this section we fix an arbitrary sequence of positive real numbers $\dN$ such that
$$
\frac{\lN^{-2}}{N} \ll  \dN \ll\epN.
$$
We let $\mR$ be a $C^{\infty}$ non-negative function of mass $1$, supported on $(-1,1)$ and we let 
$\Rd := \frac{1}{\dN} \mR\left(\frac{\cdot}{\dN}\right)$. We use the function $\Rd$ to regularize the random variable $\langle \phi, \LN \rangle$ at scale $\dN$, that is we let $\tX$ be the sum of $\langle \phi, \LN \rangle$ and of an independent random variable with density $\Rd$. Since the random variable $\tX$ almost surely differs from the original variable by at most $\dN \ll \epN$, it is enough to prove that for every $\aa$ real, we have:
$$
\P \left[  \tX \in (\aa - \epsilon_N, \aa+ \epsilon_N) \right] = \P\left[  Z \in (\aa - \epsilon_N, \aa + \epsilon_N) \right] (1 + o_N(1)).
$$ 
For simplicity of notation we take in the sequel $\aa=0$. In fact, as can be easily checked, the estimates are uniform for $\aa$
in compacts.

By regularization, $\tX$ has a  density with respect to the Lebesgue measure,
which is bounded by some function of $N$.  We write 
the ``local" probability $\P \left[  \tX \in (- \epsilon_N, \epsilon_N) \right]$ as the integral of the indicator function $\1_{[-\epsilon_N, \epsilon_N]}(t)$ against this density, and we use Parseval's identity:
\begin{equation}
\label{Parseval}
\P \left[  \tX \in (- \epsilon_N, \epsilon_N) \right] = \int_{\R} \frac{\sin(\epN \omega)}{\omega} \widehat{\Rd}(\omega) \FT(\omega) \d \omega = \int_{\R} \frac{\sin(\epN \omega)}{\omega} \widehat{\mR}(\dN \omega) \FT(\omega) \d \omega.
\end{equation}
We decompose the integral over $\R$ in three parts: low frequencies where $|\omega| \leq K$ (where $K$ is an arbitrarily large constant),
intermediate frequencies
where $K \leq |\omega| \leq \frac{1}{\dN}$,
and high frequencies where $|\omega| \geq \frac{1}{\dN}$.

1. \textbf{Low frequencies}. We let 
\begin{equation}
\label{eq-gfour}
\FTi=\Esp(e^{i\omega Z})=e^{-\omega^2 \frac{\sigma_Z^2}{2}-i\omega b_Z}
\end{equation} 
denote the characteristic function 
of $Z$.
The central limit theorem of \cite{LebSerCLT, BBNY} implies that 
$\FT$ converges pointwise to $\FTi$. For any fixed $K$, we thus have by dominated convergence:
$$
\int_{|\omega| \leq K} \frac{\sin(\epN \omega)}{\omega} \widehat{\mR}(\dN \omega) \FT(\omega) \d \omega = (1 + o_N(1)) \int_{|\omega| \leq K} \frac{\sin(\epN \omega)}{\omega} \widehat{\mR}(\dN \omega) \FTi(\omega) \d \omega .
$$
Moreover, since $\FTi$ is integrable, we have:
$$
\left| \int_{|\omega| > K} \frac{\sin(\epN \omega)}{\omega} \widehat{\mR}(\dN \omega) \FTi(\omega) \d \omega \right| = \epN o_K(1),
$$
and Parseval's identity \eqref{Parseval} holds for $Z, \FTi$ as well, so we may write:
$$
\int_{|\omega| \leq K} \frac{\sin(\epN \omega)}{\omega} \widehat{\mR}(\dN \omega) \FT(\omega) \d \omega =  \P\left[  Z \in (- \epsilon_N,  \epsilon_N) \right](1 + o_N(1))  +  \epN o_K(1).
$$

2. \textbf{Intermediate frequencies.} We have, using \eqref{FTdecayA}, the bound $\frac{\sin(\epN \omega)}{\omega}  \leq \min\left(\epN, \frac{1}{\omega}\right)$,  and a direct integration:
$$
\left|\int_{K \leq |\omega| \leq \frac{1}{\dN} } \frac{\sin(\epN \omega)}{\omega} \widehat{\mR}(\dN \omega) \FT(\omega) \d \omega \right| \leq C(\bphi, \muz, \beta) \left(\frac{1}{K} \epN + \frac{\lN^{-2}}{N} \log \frac{1}{\dN} \right).
$$
Since we have assumed that $\lN \gg N^{-1/2}$, that $\frac{\lN^{-2}}{N} \ll \dN \ll \epN$ and that $\frac{\lN^{-2} \log N}{N} \ll \epN \ll 1$, we obtain
$$
\left|\int_{K \leq |\omega| \leq \frac{1}{\dN} } \frac{\sin(\epN \omega)}{\omega} \widehat{\mR}(\dN \omega) \FT(\omega) \d \omega \right| \leq \epN \left(o_K(1) + o_N(1) \right).
$$

3. \textbf{High frequencies.} In steps 1. and 2. we have simply bounded $\widehat{\mR}$ by $1$. For $|\omega| \geq \frac{1}{\dN}$ we use the smoothness of $\mR$ and write that:
$$
 \widehat{\mR}(\dN \omega) \leq \frac{C(\mR)}{\dN^2 \omega^2}.
$$
Using \eqref{FTdecayA} again, we obtain (keeping in mind that $\frac{\lN^{-2}}{N} \ll \dN \ll \epN$):
\begin{multline*}
\left| \int_{|\omega| \geq \frac{1}{\dN} } \frac{\sin(\epN \omega)}{\omega} \widehat{\mR}(\dN \omega) \FT(\omega) \d \omega  \right| \leq C(\mR)
C(\bphi, \muz, \beta)  \int_{\frac{1}{\dN}}^{+\infty} \frac{1}{\omega} \frac{1}{\dN^2 \omega^2} \left( \frac{1}{\omega} + \frac{\lN^{-2}}{N} \right) \d\omega \\
\leq C(\mR)C(\bphi, \muz, \beta) \dN = \epN o_N(1).
\end{multline*}
In conclusion, we see that:
$$
\P \left[  \tX \in (- \epsilon_N, \epsilon_N) \right] = \P\left[  Z \in (- \epsilon_N,  \epsilon_N) \right](1 + o_N(1)) +  \epN \left(o_K(1) + o_N(1) \right).
$$
We recall that $\P\left[  Z \in (\aa - \epsilon_N, \aa + \epsilon_N) \right]$ is bounded above and below by positive constants times $\epN$, uniformly for $\aa$ in every given line segment. Taking $N \to \infty$ followed by
 $K \to \infty$ concludes the proof of Theorem \ref{theo}.

\section{Proof of Proposition \ref{prop:FTdecayA}}
\label{sec:proofdecay}
This section is devoted to the proof of Proposition \ref{prop:FTdecayA}, which goes in two steps. First, we establish a differential equation, 
see \eqref{FTomegaGomega},
for the characteristic function $\FT$. 
The homogeneous part of the differential
equation
  yields Gaussian decay. The non-homogeneous part contains a Fourier-like quantity associated to the so-called anisotropy term (see \eqref{def:Fani}
  and Section \ref{sec:Ani}), which we study next.

For the first step, we need to evaluate derivatives of various quantities,
  with respect to a parameter $\epsilon$ that is independent of $N$. The dependence in $N$ will be irrelevant because we will take $\epsilon\to 0$ with $N$ fixed.
For this reason, we 
  recall that the notation $\Ocross(\epsilon^2)$ 
  denotes a quantity that is bounded by 
  $\epsilon^2$ times a constant \textit{possibly depending on $N,\omega$} 
  (and, of course, on $\bphi, \beta, \muz$). 

In order to quantify the decay of a Fourier transform, one could try to emulate the proof of Riemann-Lebesgue's lemma. The classical integration by parts cannot be easily performed, and we replace it here by an infinitesimal change of variables. The introduction of a change of variables (infinitesimal or not) and the analysis of its effect on the energy is not new in the present context, indeed there is a vast literature around
loop equations, Ward identities, and other transportation techniques used to study fluctuations of Coulomb gases and related systems. See e.g. \cite{Shcherbina_2014,ameur2015random,zabrodin2006large,bekerman2015transport} and the papers cited above.

\subsection{Preliminaries.}
Let $\epsilon$ be a real number, we define the map $\Teps : \R^2 \to \R^2$ as follows:
\begin{equation}
\label{eq:Teps}
\Teps(x) := x + \frac{\epsilon}{\omega N} \frac{\nabla \phi(x)}{\mmuz(x)}. 
\end{equation}
We recall that by assumption the support of $\phi$ is contained inside the support $\Sigma$
of $\muz$, and that the density $\mmuz$ is bounded below by a positive constant $\cmin$ on $\Sigma$. Thus $\Teps(x)$ is well-defined for all $x$ and when $|\epsilon|$ is smaller than $\left(C(\bphi, \muz)\right)^{-1} \omega N$, the map $\Teps$ is a small (in $C^1$ norm) perturbation of the identity, hence a $C^1$-diffeomorphism of $\R^2$. We turn it into a change of variables on $(\R^2)^N$ by setting:
$$
\vTeps(x_1, \dots, x_N) := \left(\Teps(x_1), \dots, \Teps(x_N) \right).
$$
In this proof, we let $\muep$ be the push-forward of $\muz$ by the map $\Teps$. By the change of variable formula we can compute the density $\mmuep$ of $\muep$ as:
\begin{equation}
\label{muepexp}
\mmuep = \mmuz - \div\left(\mmuz \frac{\epsilon}{\omega N} \frac{\nabla \phi}{\mmuz} \right) + \Ocross(\epsilon^2) = \mmuz - \frac{\epsilon}{\omega N} \Delta \varphi + \Ocross(\epsilon^2).
\end{equation}

Let us write $\FT(\omega)$ explicitly as\footnote{Here we first choose $\YN$ to denote our “state of the system” variable, we will then perform a change of variables.}:
\begin{equation*}
\FT(\omega) = \frac{1}{\ZNbeta} \int_{(\R^2)^N} \exp\left(- \beta \F(\YN, \muz) + \vzeta(\YN) +  i \omega \langle \phi, \bYN - N \muz \rangle \right) \d\YN,
\end{equation*}
and perform the change of variables $\YN = \vTeps(\XN)$. We obtain:
\begin{equation}
\label{eq:FTafterchange}
\FT(\omega) = \frac{1}{\ZNbeta} \int_{(\R^2)^N} \exp \left(- \beta \F\left(\vTeps(\XN), \muz \right) + \vzeta(\XN)  + i \omega \langle \phi, \vTeps(\bXN) - N \muz \rangle \right) \left( \prod_{i=1}^N \det \ \D\Teps(x_i) \right) \d\XN.
\end{equation}
The confining potential $\vzeta$ has not felt the effect of the change of variables because by assumption it vanishes on $\Sigma$ and $\vTeps$ is the identity outside $\Sigma$.

We now study the three components of the integrand in \eqref{eq:FTafterchange}:
\begin{equation*}
\label{eq:TFintegrand}
\F\left(\vTeps(\XN), \muz \right), \quad i \omega \langle \phi, \vTeps(\bXN) - N \muz \rangle, \quad \prod_{i=1}^N \det\  \D\Teps(x_i).
\end{equation*}

\subsection{Effect of transport on the energy}
We study the term $\F\left(\vTeps(\XN), \muz \right)$ in \eqref{eq:FTafterchange} and prove the following expansion in $\epsilon$.
\begin{claim}
\label{claim:termAA} With $\A$ as in \eqref{def:A1}, we have:
\begin{equation}
\label{termAA}
\F\left(\vTeps(\XN), \muz \right) =  \F(\XN, \muz) + \frac{\epsilon}{N\omega} \A\left(\XN, \muz, \frac{\nabla \phi}{\mmuz} \right) +  \frac{2 \pi \epsilon}{\omega}  \langle \phi, \LN \rangle 
 + \Ocross(\epsilon^2).
\end{equation}
\end{claim}
The term $\A$ in \eqref{termAA}, whose definition is recalled in Section \ref{sec:Ani},  is the ``anisotropy'' term introduced in \cite{LebSerCLT} and studied deeply in \cite{SerfatyCLT}.  We will rely on  existing results (quoted in the Appendix) about $\A$ and its behavior under a change of variables. This import of results motivated our assumptions in
subsection \ref{subsec-settings}.

\begin{proof}[Proof of Claim \ref{claim:termAA}]
Writing $\muz =  \muep + \left(\muz - \muep\right)$ in the definition \eqref{def:FXN} of the energy yields:
\begin{equation}
\label{Ftransport1}
\F(\vTeps(\XN), \muz) = \I + \II + \III,
\end{equation}
where the terms $\I, \II, \III$ are obtained by expanding the quadratic form associated to the energy:
\begin{multline*}
\I := \F\left(\vTeps(\XN), \Teps \# \muz\right), \quad \II := \hal \iint - \log |x-y| \left(N \d\muep - N \d\muz \right)^{\otimes 2}(x,y), 
\\ \III :=  \iint - \log |x-y| \left(\vTeps(\XN) - N \d\muep \right)(x) \left(N \d\muep - N \d\muz \right)(y).
\end{multline*}

\begin{itemize}
  \item For $\I$, using \eqref{eq-firstpartA1}
    of Lemma~\ref{lem:effect} with $\mu = \muz$, $\psi = \frac{\nabla \varphi}{\mmuz}$ and $t = \frac{\epsilon}{N \omega}$, we obtain:
\begin{equation}
\label{IdansA}
\I := \F\left(\vTeps(\XN), \Teps \# \muz\right) = \F(\XN, \muz) + \frac{\epsilon}{N\omega}  \A\left(\XN, \muz, \frac{\nabla \phi}{\mmuz} \right)  + \Ocross(\epsilon^2).
\end{equation}
\item For $\II$, identity \eqref{muepexp} shows that it is of second order in $\epsilon$:
\begin{equation}
\label{IIdansA}
\II = \Ocross(\epsilon^2).
\end{equation}
\item For $\III$, using \eqref{muepexp} again we find
\begin{equation*}
\III =  \iint - \log|x-y|  \left( - \frac{\epsilon}{\omega} \Delta \varphi(y) \right) \d y \left( \sum_{i=1}^N \delta_{\Teps(x_i)} - N (\Teps \# \muz)\right)(\dx) + \Ocross(\epsilon^2).
\end{equation*}
The integral over $y$ can be eliminated using the following identity, valid for $f$ in $C^2_c(\R^2)$,
\begin{equation}
\label{identity}
\int_{\R^2} - \log|x-y|  \Delta f(y) \d y = - 2\pi f(x).
\end{equation}
Changing variables by $\Teps$ in the integral over $x$, we conclude that
$$
\III = 2 \pi  \frac{\epsilon}{\omega} \int \phi \circ \Teps(x) \left( \sum_{i=1}^N \delta_{x_i} - N \muz\right)(\dx) = \frac{2 \pi \epsilon}{\omega}  \langle \phi \circ \Teps, \LN \rangle.
$$
Since $\phi$ is differentiable and $\Teps = \id + \Ocross(\epsilon)$ we have $\phi \circ \Teps = \phi + \Ocross(\epsilon)$ and thus
\begin{equation}
\label{IIIdansA}
\III = \frac{2 \pi \epsilon}{\omega}  \langle \phi, \LN \rangle + \Ocross(\epsilon^2).
\end{equation}
\end{itemize}
Combining \eqref{IdansA}, \eqref{IIdansA}, \eqref{IIIdansA} proves the claim.
\end{proof}

\subsection{Effect of transport on the fluctuations}
We study the term $i \omega \langle \phi, \vTeps(\bXN) - N \muz \rangle$ in \eqref{eq:FTafterchange} and prove the following expansion in $\epsilon$:
\begin{claim} We have:
\label{claim:termBB}
\begin{equation}
\label{eq:termBB}
i \omega \langle \phi, \vTeps(\bXN) - N \muz \rangle = i \omega  \langle \phi, \LN \rangle + i \epsilon \int_{\R^2} |\nabla \phi(x)|^2 \d x  + \frac{i \epsilon}{N} \llangle  \frac{|\nabla \phi|^2}{\mmuz}, \LN \rrangle + \Ocross(\epsilon^2)
\end{equation}
\end{claim}
\begin{proof}[Proof of Claim \ref{claim:termBB}]
The following identity holds:
\begin{equation*}
\langle \phi, \vTeps(\bXN) - N \muz \rangle = \langle \phi, \bXN - N \muz \rangle + \sum_{i=1}^N \left( \phi \circ  \Teps(x_i) - \phi(x_i) \right).
\end{equation*}
From the definition \eqref{eq:Teps} of the map $\Teps$ and 
  Taylor's expansion
we obtain that
$$
\phi \circ \Teps(x) = \phi\left(x + \frac{\epsilon}{N \omega} \frac{\nabla \phi(x)}{\mmuz(x)} \right) = \varphi(x) + \frac{\epsilon}{N \omega} \nabla \phi(x) \cdot \frac{\nabla \phi(x)}{\mmuz(x)} + \Ocross(\epsilon^2).
$$
We can thus write $\langle \phi, \vTeps(\bXN) - \muz \rangle$ as
$$
\langle \phi, \vTeps(\bXN) - N \muz \rangle =  \langle \phi, \LN \rangle + \frac{\epsilon}{N \omega} \sum_{i=1}^N \frac{|\nabla \phi(x_i)|^2}{\mmuz(x_i)} + \Ocross(\epsilon^2).
$$
We may compare the discrete sum to an integral against $\muz$ by introducing a new fluctuation term:
$$
\frac{\epsilon}{N \omega} \sum_{i=1}^N  \frac{|\nabla \phi(x_i)|^2}{\mmuz(x_i)} = \frac{\epsilon}{\omega} \int_{\R^2} \frac{|\nabla \phi(x)|^2}{\mmuz(x)} \mmuz(x) \d x + \frac{\epsilon}{N \omega} \llangle \frac{|\nabla \phi|^2}{\mmuz},\LN \rrangle, 
$$
and we finally obtain:
\begin{equation*}
i \omega \langle \phi, \vTeps(\bXN) - N \muz \rangle = i \omega  \langle \phi, \LN \rangle + i \epsilon \int_{\R^2} |\nabla \phi(x)|^2 \d x  + \frac{i \epsilon}{N} \llangle \frac{|\nabla \phi|^2}{\mmuz},\LN \rrangle + \Ocross(\epsilon^2),
 \end{equation*}
which proves the claim.
\end{proof}

\subsection{Computation of the Jacobian}
We study the term $\prod_{i=1}^N \det\ \D\Teps(x_i)$ in \eqref{eq:FTafterchange} and prove the following expansion in $\epsilon$:
\begin{claim} We have:
\label{claim:termCC}
\begin{equation}
\label{eq:termCC}
\prod_{i=1}^N \det\ \D\Teps(x_i) = 1 + \frac{\epsilon}{\omega N} \int_{\R^2} \Delta \varphi(x) \log \mmuz(x) \d x + \frac{\epsilon}{\omega N} \llangle \div \left( \frac{\nabla \varphi}{\mmuz} \right), \LN \rrangle + \Ocross(\epsilon^2).
\end{equation}
\end{claim}
\begin{proof}[Proof of Claim \ref{claim:termCC}]
  The proof follows a computation from
    \cite{LebSerCLT}.
  We rewrite the left-hand side of \eqref{eq:termCC} as $\exp\left(\sum_{i=1}^N \log \det \D\Teps(x_i)\right)$, and compare the summand with an integral against $\muz$ by introducing a new fluctuation term: 
  \begin{equation}
    \label{eq-sta1}
\sum_{i=1}^N \log \det \ \D\Teps(x_i) = N \int_{\R^2} \log \det \ \D\Teps(x) \d\muz(x) + \llangle \log \det \ \D\Teps, \LN \rrangle.
\end{equation}
By the change of variables formula, we have that
$\det \ \D\Teps(x) = \frac{\mmuz(x)}{\mmuep \circ \Teps(x)}$ and therefore
$$
\int_{\R^2} \log \det \  \D\Teps(x) \d \muz(x) = \int_{\R^2} \log \mmuz(x) \d\muz(x) - \int_{\R^2} \log \mmuep (x) \d \muep(x).
$$
Using \eqref{muepexp} and a first order Taylor
expansion in the last display, we obtain that
\begin{equation}
  \label{eq-sta2}
\int_{\R^2} \log \det \  \D\Teps(x) \d\muz(x) = \frac{\epsilon}{N \omega} \int_{\R^2} \Delta \varphi(x) \log \mmuz(x) \dx + \Ocross(\epsilon^2).
\end{equation}
On the other hand, using a Taylor's expansion for $\det$
around the identity matrix we have:
$$
\log \det \ \D\Teps (x) = \frac{\epsilon}{N \omega} \div \left( \frac{\nabla \varphi}{\mmuz} \right)(x) + \Ocross(\epsilon^2).
$$
Combining the last display with \eqref{eq-sta1} and \eqref{eq-sta2}, we obtain
\eqref{eq:termCC} and complete the proof of the claim.
\end{proof}

\subsection{Combining the effects of the transport}
\label{sec:combineeffects}
Using Claims
  \ref{claim:termAA}, \ref{claim:termBB} and \ref{claim:termCC} in \eqref{eq:FTafterchange}, we obtain:
\begin{multline*}
\FT(\omega) = \frac{1}{\ZNbeta} \int \exp\left(- \beta \F(\XN, \muz) + \vzeta(\XN)  +  i\omega \langle \phi, \LN \rangle \right)
\\
 \times \left(1  - \frac{\beta \epsilon}{N\omega} \A\left(\XN, \muz, \frac{\nabla \phi}{\mmuz} \right) -  \frac{2 \pi \beta \epsilon}{\omega}  \langle \phi, \LN \rangle 
 + \Ocross(\epsilon^2)  \right)
 \\
\times \left( 1 + i \epsilon \int_{\R^2} |\nabla \phi(x)|^2 \d x  + \frac{i \epsilon}{N} \llangle  \frac{|\nabla \phi|^2}{\mmuz}, \LN \rrangle + \Ocross(\epsilon^2) \right)
\\
 \times \left(1 + \frac{\epsilon}{\omega} \int_{\R^2} \Delta \varphi(x) \log \mmuz(x) \d x + \frac{\epsilon}{\omega N} \llangle \div \left( \frac{\nabla \varphi}{\mmuz} \right), \LN \rrangle + \Ocross(\epsilon^2) \right)  \d\XN.
\end{multline*}

Since the left side in the last display is independent of
$\epsilon$, identifying the terms of order $1$ in $\epsilon$ gives that:
\begin{multline*}
\frac{1}{\ZNbeta} \int \exp\left(- \beta \F(\XN, \muz) + \vzeta(\XN)  +  i\omega \langle \phi, \LN \rangle \right) \left(  - \frac{\beta}{N\omega} \A\left(\XN, \muz, \frac{\nabla \phi}{\mmuz} \right) -  \frac{2 \pi \beta}{\omega}  \langle \phi, \LN \rangle  \right. \\
+ i \int_{\R^2} |\nabla \phi(x)|^2 \d x + \frac{i}{N} \llangle  \frac{|\nabla \phi|^2}{\mmuz}, \LN \rrangle \\
+\left. \frac{1}{\omega} \int_{\R^2} \Delta \varphi(x) \log \mmuz(x) \d x + \frac{1}{\omega N} \llangle \div \left( \frac{\nabla \varphi}{\mmuz} \right), \LN \rrangle \right) \d\XN = 0,
\end{multline*}
which can be seen as the expectation (under $\PNbeta$) of the quantity $\exp\left( i\omega \langle \phi, \LN \rangle\right)$ multiplied by some terms. This identity is analoguous to e.g. \cite[Prop 7.3]{BBNYA}, but this time with a complex-valued function $f$ (here $i \varphi$). We now make the following observations:

1. The expectation of $-\frac{2\pi\beta}{\omega} \langle \phi, \LN \rangle$ yields a multiple of the derivative of $\FT(\omega)$, indeed:
$$
-\frac{2\pi\beta}{\omega} \Esp\left[ \exp\left( i\omega \langle \phi, \LN \rangle\right) \langle \phi, \LN \rangle \right] = -\frac{2\pi\beta}{\omega} \frac{1}{i} \FT'(\omega) = \frac{2 i \pi\beta}{\omega} \FT'(\omega).
$$
We note in passing that the choice of the transport $T_\epsilon$, see
  \eqref{eq:Teps}, was geared toward obtaining the derivative
  $\FT'(\omega)$ in the last display. The fact that we are working in
  dimension $d=2$ played a crucial role in the design of
  the transport, via the identity \eqref{identity}.

2. The term $i \int_{\R^2} |\nabla \phi(x)|^2 \d x = i \| \phi \|_{H^1}^2 $ is a constant, and so is the term $\frac{1}{\omega} \int_{\R^2} \Delta \varphi(x) \log \mmuz(x) \d x$.

3. For the fluctuation terms, using the fact that
``fluctuations are bounded in $L^1(\P)$'' as expressed in Lemma~\ref{lem:FluctuationsBounded}, we can write:
\begin{equation}
  \label{eq-sat3}
\left| \Esp\left[\exp\left( i\omega \langle \phi, \LN \rangle\right)  \frac{i}{N} \llangle   \frac{|\nabla \phi|^2}{\mmuz}, \LN \rrangle \right] \right| \leq \frac{1}{N} \Esp\left[ \left| \llangle  \frac{|\nabla \phi|^2}{\mmuz}, \LN \rrangle \right| \right] \leq  C(\bphi, \beta, \muz) \frac{1}{N}  \lN^{-2} ,
\end{equation}
where $C(\bphi, \beta, \muz)$ is some constant (depending on the test
function $\bphi$, on $\beta$ and on $\muz$). Indeed, since $\phi$ lives at scale $\lN$, the gradient $\nabla \phi$ is of order $\llN$, the function $\frac{|\nabla \phi|^2}{\mmuz}$ is bounded by a constant (depending on $\mmuz$) times $\llN^{-2}$ and every additional derivative loses a factor $\llN^{-1}$. 
The \textit{a priori} bound of 
Lemma \ref{lem:FluctuationsBounded}
states that the fluctuation of 
$\langle \frac{|\nabla \phi|^2}{\mmuz},\LN\rangle$
is of order $\llN^{-2}$. Similarly, we have:
\begin{equation}
  \label{eq-sat4}
\left| \Esp\left[\exp\left( i\omega \langle \phi, \LN \rangle\right)  \frac{1}{\omega N} \llangle \div \left( \frac{\nabla \varphi}{\mmuz} \right), \LN \rrangle   \right] \right| \leq \frac{1}{\omega N}  \Esp \left[  \left| \llangle \div \left( \frac{\nabla \varphi}{\mmuz} \right), \LN \rrangle   \right| \right] \leq  C(\bphi, \beta, \muz) \frac{1}{\omega N}  \lN^{-2}.
\end{equation}
Since we are interested in the \textit{decay} of $\FT(\omega)$ for 
$|\omega|>1$, the bound \eqref{eq-sat3}
is the dominant one. We note in passing that
in order to apply Lemma \ref{lem:FluctuationsBounded}, we need both $\frac{|\nabla \phi|^2}{\mmuz}$ and $ \div \left( \frac{\nabla \varphi}{\mmuz} \right)$ to be of class $C^3$, which is guaranteed when $\phi$ itself is of class $C^5$ and $\mmuz$ of class $C^3$, hence our regularity assumptions - which we have not tried to optimize.

4. Concerning the term $ - \frac{\beta}{N\omega} \A\left(\XN, \muz, \frac{\nabla \phi}{\mmuz} \right)$, in  \eqref{def:A1} of
Section \ref{sec:Ani} we decompose $\A$ into: 
$$
\A\left(\XN, \muz, \frac{\nabla \phi}{\mmuz}\right) = \Ani\left(\XN, \muz, \frac{\nabla \phi}{\mmuz} \right) + \frac{1}{4} \sum_{i=1}^N \div \left(\frac{\nabla \phi}{\mmuz} \right)(x_i).
$$
The sum in the right-hand side can be compared to a continuous integral, and we obtain, after an integration by parts, that
\begin{equation*}
\sum_{i=1}^N \div \left(\frac{\nabla \phi}{\mmuz} \right)(x_i) = N \int_{\R^2} \Delta \varphi(x) \log \mmuz(x) \d x + \llangle \div \left(\frac{\nabla \phi}{\mmuz}\right),  \LN \rrangle,
\end{equation*}
with the same fluctuation term as the one obtained in Claim \ref{claim:termCC}.
(The first term in the right hand side of the last display 
  contributed to the mean of $Z$ in Theorem \ref{theo}.)
To summarize, we have:
\begin{equation*}
- \frac{\beta}{N\omega} \A\left(\XN, \muz, \frac{\nabla \phi}{\mmuz} \right) = - \frac{\beta}{N\omega} \Ani\left(\XN, \muz, \frac{\nabla \phi}{\mmuz} \right) - \frac{\beta}{4 \omega} \int_{\R^2} \Delta \varphi(x) \log \mmuz(x) \d x  + C(\bphi, \beta, \muz) \frac{1}{\omega N}  \lN^{-2}.
\end{equation*}

We may thus write, in conclusion:
\begin{multline}
\label{FTomegaGomega}
\left( i \| \phi \|_{H^1}^2 + \left(\frac{1}{\omega} - \frac{\beta}{4 \omega}\right) \int_{\R^2} \Delta \varphi(x) \log \mmuz(x) \d x\right) \FT(\omega) + \frac{2i \pi \beta}{\omega} \FT'(\omega) \\
=  \frac{\beta}{N\omega} \G(\omega) + C(\bphi, \beta, \muz) \frac{1}{N}  \lN^{-2},
\end{multline}
where $\G(\omega)$ is defined as $\Esp \left[ e^{i\omega \langle \phi, \LN \rangle} \Ani\left(\XN, \muz, \frac{\nabla \phi}{\mmuz} \right) \right]$, i.e.\footnote{Here we simply switch back to $\YN$ as the variable in the integral.}: 
\begin{equation}
\label{def:Fani}
\G(\omega) := \frac{1}{\ZNbeta} \int_{(\R^2)^N} \exp\left(- \beta \F(\YN, \muz) + \vzeta(\YN)  + i \omega \langle \phi, \bYN - N \muz \rangle \right)  \Ani\left(\YN, \muz, \frac{\nabla \phi}{\mmuz} \right) \d\YN.
\end{equation}

\begin{remark}
If the right-hand side of \eqref{FTomegaGomega} were zero, we would obtain a simple ODE for $\FT(\omega)$ of the form
\begin{equation*}
\FT'(\omega) = \left(- \frac{\omega \| \phi \|_{H^1}^2 }{2 \pi \beta} + \frac{i}{2\pi} \left( \frac{1}{\beta} - \frac{1}{4} \right) \int_{\R^2} \Delta \varphi(x) \log \mmuz(x) \d x \right) \FT(\omega),
\end{equation*}
whose solution (since $\FT(0) = 1$) would indeed coincide with the Fourier transform of the limiting distribution for fluctuations of linear statistics as obtained in \cite{LebSerCLT, BBNY}.
\end{remark}

Using an \textit{a priori} bound on $\Ani$, we can see that $|\G(\omega)|$ is of order $N$. This is not enough for our purposes, which is why we need to go further and study the decay of $\G(\omega)$ itself.
This is established in the next section.
\subsection{Additional decay estimate}
\begin{lemma} There exists a constant $C(\bphi,\muz,\beta)$ so that, 
  for $|\omega| > 1$,
\label{lem:adddecay}
\begin{equation}
\label{decayG}
\left| \G(\omega) \right|  \leq C(\bphi, \muz, \beta) \left( \frac{N}{\omega} + \lN^{-2} \right).
\end{equation}
\end{lemma}
\begin{proof}[Proof of Lemma \ref{lem:adddecay}]
We use the same strategy as above, and perform the same change of variables, namely $\YN = \vTeps(\XN)$ in the right-hand side of \eqref{def:Fani}. We get:
\begin{multline}
\label{gomegachange}
\G(\omega) = \frac{1}{\ZNbeta} \int_{(\R^2)^N} \exp \left(- \beta \F\left(\vTeps(\XN), \muz \right) + \vzeta(\XN)  + i \omega \langle \phi, \vTeps(\bXN) - N \muz \rangle \right) \\
\times \Ani\left(\vTeps(\XN), \muz, \frac{\nabla \phi}{\mmuz} \right)  \left( \prod_{i=1}^N \det \ \D\Teps(x_i)\right) \d\XN.
\end{multline}
The effect of this change of variables has been previously computed for all the terms except for the anisotropy itself. 

\begin{claim}[Effect of the change of variables on the anisotropy]
\label{claim:changeani}
When $\YN = \vTeps(\XN)$, we have:
\begin{equation}
\label{changeani}
\left|\Ani\left(\YN, \muz, \frac{\nabla \phi}{\mmuz}\right) - \Ani\left(\XN, \muz, \frac{\nabla \phi}{\mmuz}\right)\right|  \leq 
C(\bphi, \beta, \muz)  \frac{\epsilon}{\omega} \frac{\lN^{-3}}{N^{1/2}} \left( \EnergiePoints + N \lN^2 \right) + \Ocross(\epsilon^2).
\end{equation}
\end{claim}
The quantity $\EnergiePoints$ will be introduced in 
Section \ref{sec:enerpoints} as the sum of the electric field
  $\Ele$ and the number of points $\Points$ on the support of
$\phi$, but
   its precise definition is irrelevant for the moment. It suffices to keep in mind that $\EnergiePoints$ is of order $N \llN^2$ in a sense made precise below.  We postpone the proof of Claim \ref{claim:changeani} to Section \ref{sec:proofchangeani}. We then argue as in Section \ref{sec:combineeffects}. We obtain: 
\begin{multline*}
\G(\omega) =
\frac{1}{\ZNbeta} \int \exp\left(- \beta \F(\XN, \muz) + \vzeta(\XN)  + i\omega \langle \phi, \LN \rangle \right)
\\
 \times \left(1  - \frac{\beta \epsilon}{N\omega} \A\left(\XN, \muz, \frac{\nabla \phi}{\mmuz} \right) -  \frac{2 \pi \beta \epsilon}{\omega}  \langle \phi, \LN \rangle 
 + \Ocross(\epsilon^2)  \right)
 \\
\times \left( 1 + i \epsilon \int_{\R^2} |\nabla \phi(x)|^2 \d x  + \frac{i \epsilon}{N} \llangle  \frac{|\nabla \phi|^2}{\mmuz}, \LN \rrangle + \Ocross(\epsilon^2) \right)
\\
 \times \left(1 + \frac{\epsilon}{\omega} \int_{\R^2} \Delta \varphi(x) \log \mmuz(x) \d x + \frac{\epsilon}{\omega N} \llangle \div \left( \frac{\nabla \varphi}{\mmuz} \right), \LN \rrangle + \Ocross(\epsilon^2) \right) \\
\times 
\left( \Ani\left(\XN, \muz, \frac{\nabla \phi}{\mmuz}\right) + C(\bphi, \beta, \muz)  \frac{\epsilon}{\omega} \frac{\lN^{-3}}{N^{1/2}} \left( \EnergiePoints + N \lN^2 \right) + \Ocross(\epsilon^2) \right)
  \d\XN.
\end{multline*}
Identifying again the terms of order $1$ in $\epsilon$, we find that
\begin{multline*}
\frac{1}{\ZNbeta} \int \exp\left(- \beta \F(\XN, \muz) + \vzeta(\XN)  + i\omega \langle \phi, \LN \rangle \right) 
 \Bigg[ \Ani\left(\XN, \muz, \frac{\nabla \phi}{\mmuz}\right) \\
   \times \Big[ - \frac{\beta}{N\omega} \A\left(\XN, \muz, \frac{\nabla \phi}{\mmuz} \right) -  \frac{2 \pi \beta}{\omega}  \langle \phi, \LN \rangle \\
+ i \int_{\R^2} |\nabla \phi(x)|^2 \d x  + \frac{i}{N} \llangle  \frac{|\nabla \phi|^2}{\mmuz}, \LN \rrangle \\
+ \frac{1}{\omega} \int_{\R^2} \Delta \varphi(x) \log \mmuz(x) \d x + \frac{1}{\omega N} \llangle \div \left( \frac{\nabla \varphi}{\mmuz} \right), \LN \rrangle \Big] \\
+ C(\bphi, \beta, \muz)  \frac{1}{\omega} \frac{\lN^{-3}}{N^{1/2}} \left( \EnergiePoints + N \lN^2 \right) \Bigg] \d\XN = 0.
\end{multline*}
We use the following a priori bounds, valid in $L^2(\PNbeta)$ (see Section \ref{sec:aprioribounds}):
\begin{eqnarray*}
\left| \Ani\left(\XN, \muz, \frac{\nabla \phi}{\mmuz} \right) \right| & \leq & C(\bphi, \muz, \beta) N 
\quad (\mbox{\rm Lemma \ref{lem:aprioribounds}})\\
\left| \langle \phi, \LN \rangle \right| & \leq  & C(\bphi, \muz, \beta)
\quad (\mbox{\rm Lemma \ref{lem:FluctuationsBounded}})\\
\left| \llangle  \frac{|\nabla \phi|^2}{\mmuz}, \LN \rrangle \right|  & \leq & C(\bphi, \muz, \beta) \lN^{-2}\quad (\mbox{\rm Lemma \ref{lem:FluctuationsBounded}})\\
\left| \llangle \div \left( \frac{\nabla \varphi}{\mmuz} \right), \LN \rrangle \right|   & \leq  & C(\bphi, \muz, \beta) \lN^{-2}\quad (\mbox{\rm Lemma \ref{lem:FluctuationsBounded}})\\
\EnergiePoints & \leq & C(\muz, \beta) N \lN^2 \quad (\mbox{\rm Lemma \ref{lem:aprioribounds}})
\end{eqnarray*}
and, keeping only the dominant terms, we obtain that:
\begin{equation*}
\left| \G(\omega) \right|  \leq C(\bphi, \muz, \beta) \left( \frac{N}{\omega} + \lN^{-2} \right),
\end{equation*}
which concludes the proof of Lemma \ref{lem:adddecay}.
\end{proof}

Substituting the estimate from
  Lemma \ref{lem:adddecay} into \eqref{FTomegaGomega}, 
we obtain a differential equation for $\FT$ of the form:
\begin{equation*}
(iA + B) \FT(\omega) + \frac{iC}{\omega} \FT'(\omega) = \frac{D_1}{\omega^2} + \frac{D_2 \lN^{-2}}{N},
\end{equation*}
where $A, C$ are positive constants (depending on $\beta, \bphi$), $D_1, D_2$ are bounded (depending on $\beta, \bphi, \muz$) and $B$ is real
(depending on $\beta,\bphi,\muz$). Solving this equation yields:
\begin{equation*}
|\FT(\omega)| \leq C(\bphi, \muz, \beta)  \left( \frac{1}{\omega^2} + \frac{\lN^{-2}}{N} \right),
\end{equation*}
which concludes the proof of Proposition \ref{prop:FTdecayA}.

\appendix
\section{Electric fields and anisotropy}
Let $\mu$ be a probability measure with a bounded density and $\XN$ a $N$-tuple of points in $\R^2$.

In this section, 
we denote by $|f|_{\kk}$ the norm of the $\kk$-th derivative of $f$ 
(with $|f|_0$ denoting the supremum norm of $f$).
We recall that $K$ denotes a universal constant, 
dependent only on $\muz,\beta$ and in particular independent of 
$N,\omega,\psi$, whose value may change from line to line. We also recall the 
definition of $\Ocross(t)$, see subsection \ref{subsec-notation}.

\subsection{The electric field formalism}
\label{sec:electricfield}
We define the \textit{electric field} $\nhmu$ associated with $\mu$ and $\XN$ as
$$
\nhmu (z) =  \left( \sum_{i=1}^N - \nabla \log |z-x_i| - N \int_{\R^2} - \nabla \log |z -x| \d \mu(x) \right).
$$
This vector field is in $\Lploc$ for every $p < 2$ but fails to be in $L^2$ because of its singularity around each point charge.  Let $\veta = (\eta_1, \dots, \eta_N)$ be a vector of positive numbers, we define the \textit{truncated electric field} $\nhmue$ as: 
\begin{equation}
\label{def:Eleta}
\nhmue(z) := \nhmu(z) - \sum_{i=1}^N \nabla \fetai(z-x_i),
\end{equation}
where $\fetai(z) := \left(\log \eta_i - \log |z| \right)_+$. We 
refer to e.g. \cite[Section 2.2]{LebSerCLT} and the references therein for a more detailed discussion of this procedure. It has the effect of canceling the contribution to the electric field coming from the particle $x_i$ in the disk of radius $\eta_i$ around $x_i$, thus effectively truncating the singularity around each charge.

A particularly convenient choice of the truncation parameter is the nearest-neighbor distances. For any $i = 1, \dots, N$ we let $\r_i$ be the following quantity:
\begin{equation}
\label{def:vri}
\r_i := \frac{1}{4} \min \left( \frac{1}{\sqrt{N}}, \min_{j \neq i} |x_i - x_j| \right),
\end{equation}
and we let $\vr := (\r_1, \dots, \r_N)$. 
 
\subsection{Electric energy (and number of points)}
\label{sec:enerpoints}
For $\Omega \subset \R^2$,\footnote{We note that in the application in Section \ref{sec:proofdecay} we always
take $\Omega$ to be the support of $\phi$} we introduce the notation $\Ele(\Omega)$ to denote the ``electric field energy'' in $\Omega$
\begin{equation}
\label{def:Ele}
\Ele(\Omega) := \int_{\Omega} |\El_{\vr}|^2.
\end{equation}

We also let $\Points$ be the number of points of the configuration $\XN$ in (an
$N^{-1/2}$-neighborhood of) $\Omega$.

When working with a test function $\varphi$ compactly supported in a domain of characteristic size $\lN$, we  write $\EnergiePoints$ for the sum of the electric energy and the number of points on the support of $\varphi$. This quantity is typically of order $N \lN^2$, see Lemma \ref{lem:aprioribounds}.

\subsection{The anisotropy}
\label{sec:Ani}
The
``anisotropy''
term was introduced in \cite{LebSerCLT} as the first order effect of a transport map (applied both to the particles and the background measure) to the interaction energy. We will also quote results from \cite[Section 4]{SerfatyCLT}), where this quantity is studied extensively. 

When $\mu$ has a bounded density and $\psi$ is a $C^1$ vector field, we define $\Ani(\XN, \mu, \psi)$ as:
\begin{equation}
\label{def:Ani}
\Ani(\XN, \mu, \psi) := \hal \lim_{\eta \to 0^+} \frac{1}{2\pi} \int_{\R^2} \llangle \nhmue , \left(2 \D \psi - (\div \psi) \id \right) \nhmue \rrangle,
\end{equation}
where $\nhmue$ is the truncated electric field defined above, with truncation $\veta := (\eta, \dots, \eta)$.

Note that in \cite{SerfatyCLT} however, the ``anisotropy''
is rather defined as:
\begin{equation}
\label{def:A1}
\A(\XN, \mu, \psi) := \Ani(\XN, \mu, \psi) + \frac{1}{4} \sum_{i=1}^N \div \psi(x_i).
\end{equation}
One advantage is that $\A$ admits an alternative
non-asymptotic expression 
(no need to take a limit $\eta \to 0$ as in \eqref{def:Ani}), 
see \cite[(4.11)]{SerfatyCLT}.
Generally, $\A$ is
the natural quantity to consider in several questions. However, for our purposes, we sometimes prefer to consider the two terms in the right-hand side of \eqref{def:A1} separately.

\subsection{Effect of a transport map}
The following results are direct consequences of \cite[Proposition 4.3]{SerfatyCLT}, upon noting that $\Xi(t)$ of the latter satisfies
  that $\F\left( (\id + t \psi) (\XN), (\id + t \psi) \# \mu\right)-\Xi(t)$
  is independent of $t$, substituting
  \cite[(4.8)]{SerfatyCLT} into \cite[(4.9),(4.10)]{SerfatyCLT}, and using
  \cite[(4.11)]{SerfatyCLT}.
\begin{lemma}[Effect of a transport map]
\label{lem:effect}
Let $\psi$ be in $C^2(\R^2, \R^2)$ with compact support. We have, as $t \to 0$:
\begin{equation}
  \label{eq-firstpartA1}
\F\left( (\id + t \psi) (\XN), (\id + t \psi) \# \mu\right) - \F\left(\XN, \mu \right) = t \A(\XN, \mu, \psi) + \Ocross(t^2).
\end{equation}
Further,
\begin{multline}
  \label{eq-secondpartA1}
  \left|\A\left( (\id + t \psi)(\XN),  (\id + t \psi) \# \mu, \psi \right) 
  -  \A\left( \XN,\mu, \psi \right)\right|  \\
  \leq K  t \left( |\psi|_{\1}^2 + |\psi|_{\0} |\psi|_{\2} + |\psi|_{\1} |\psi|_{\2} N^{-1/2} \right) \left| \F\left(\XN, \mu \right) + \hal \# I_N \log N \right|+\Ocross(t^2),
\nonumber
\end{multline}
where $ \# I_N$ denotes the number of points of $\XN$ in (a $N^{-1/2}$-neighborhood of) the support of $\psi$.

We can easily translate both results in terms of $\Ani$ (instead of $\A$) by writing
\begin{equation}
  \label{eq-sat5}
\F\left( (\id + t \psi) (\XN), (\id + t \psi) \# \mu\right) - \F\left(\XN, \mu \right) = t \Ani(\XN, \mu, \psi) + t \sum_{i=1}^N \div \psi(x_i) +  \Ocross(t^2).
\end{equation}
and
\begin{multline}
  \label{eq-sat6}
\left|\Ani\left( (\id + t \psi)(\XN), \psi,  (\id + t \psi) \# \mu \right) -  \Ani\left( \XN, \psi, \mu \right)\right| \\
\corO{\leq K}  t \left( \left( |\psi|_{\1}^2 + |\psi|_{\0} |\psi|_{\2} + |\psi|_{\1} |\psi|_{\2} N^{-1/2} \right) \left| \F\left(\XN, \mu \right) + \hal \# I_N \log N \right)+ \# I_N |\psi|_{\2} |\psi|_{\0}\right| .
\end{multline}
\end{lemma}
\corO{(When applying \cite{SerfatyCLT}, one may wonder whether one needs
to consider $I_N(\XN)$ or $I_N((\id + t \psi)(\XN))$. However, since 
we only care about the asymptotics as $t\to 0$, 
and since all points $x_i$ are translated 
at most by
$t|\psi(x)|$ and are only translated within the support of $\psi$, 
there is no difference which of the two we consider.)}
In our application, $|\psi|_{\0}$ is of order $\llN^{-1}$ and each derivative loses a factor $\llN^{-1} \ll N^{1/2}$, so in the second comparison, the factor $|\psi|_{\1}^2 + |\psi|_{\0} |\psi|_{\2} + |\psi|_{\1} |\psi|_{\2} N^{-1/2}$ is of order $\llN^{-2}$.

\subsection{Proof of Claim \ref{claim:changeani}}
\label{sec:proofchangeani}.
\begin{proof}[Proof of Claim \ref{claim:changeani}]
  \corO{We again denote by} $\muep$
  the push-forward of $\muz$ by $\Teps$, and \corO{we} let $\YN = \vTeps(\XN)$. 
To prove the claim, we need to compare $\Ani\left(\XN, \muz, \frac{\nabla \phi}{\mmuz}\right)$ and $\Ani\left(\YN, \muz, \frac{\nabla \phi}{\mmuz}\right)$.

\corO{We} first compare $\Ani\left(\XN, \muz, \frac{\nabla \phi}{\mmuz}\right)$ with $\Ani\left(\YN, \muep, \frac{\nabla \phi}{\mmuz}\right)$. Applying \corO{\eqref{eq-sat6} of}
the second part of Lemma \ref{lem:effect}, with $\mu = \muz$, $\psi = \frac{\nabla \phi}{\mmuz}$ and $t = \frac{\epsilon}{N\omega}$, and using the scaling properties of $\phi$ and its derivatives we get:
\begin{equation}
\label{anitransporteB}
\left|\Ani\left(\YN, \muep, \frac{\nabla \phi}{\mmuz}\right) - \Ani\left(\XN, \muz, \frac{\nabla \phi}{\mmuz}\right) \right| \leq C(\bphi, \beta, \muz)  \frac{\epsilon \lN^{-4}}{N \omega} \EnergiePoints,
\end{equation}
where we use the notation $\EnergiePoints$ as introduced in Section \ref{sec:enerpoints}. It remains to compare $\Ani\left(\YN, \muep, \frac{\nabla \phi}{\mmuz}\right)$ (that we will denote here by $\AniYe$ for brevity) and $\Ani\left(\YN, \muz, \frac{\nabla \phi}{\mmuz}\right)$ (denoted below by $\AniYz$). The points stay the same but the reference measure changes. It will be convenient to use some notation:
\begin{itemize}
\item Let $\EYz$ denote the electric field corresponding to the charges in $\YN$ and the background of density $\mmuz$, namely
\begin{equation*}
\EYz(z) := \int_{\R^2} - \nabla \log |z - x| \left(\bYN - N \muz \right)(\d x).
\end{equation*}
\item Let $\EYe$ denote the electric field corresponding to the charges in $\YN$ and \textit{the modified background} $\muep$, namely:
\begin{equation*}
\EYe(z) := \int_{\R^2} - \nabla \log |z - x| \left(\bYN - N\muep \right)(\dx),
\end{equation*}
\item  We let $\EYze$, $\EYee$ be the corresponding truncated fields.
\end{itemize}
Using the expansion of $\muep$ stated in \eqref{muepexp}, the difference between $\EYz$ and $\EYe$ is fairly easy to compute. We get: 
\begin{equation} \label{EYzvsEYe}
\EYze(z)  -\EYee(z) = \EYz(z) - \EYe(z) = 2\pi \frac{\epsilon}{\omega} \nabla \varphi(z) + \Ocross(\epsilon^2),
\end{equation}
where we have used again the identity \eqref{identity}.

Going back to the definition \eqref{def:Ani} of $\Ani$ we see that (with $\psi = \frac{\nabla \phi}{\mmuz}$ for brevity)
\begin{multline*}
2\left(\AniYe - \AniYz\right) \\ 
=  \lim_{\eta \to 0^+} \frac{1}{2\pi} \int_{\R^2} \llangle \EYze  \left(2 \D \psi - \left(\div \psi \right) \id \right) \EYee - \EYze \rrangle + \llangle \EYee - \EYze , \left(2 \D \psi - \left(\div \psi \right) \id \right) \EYee \rrangle.
\end{multline*}
Both terms are bounded similarly. Hölder's inequality yields:
\begin{multline*}
 \int_{\R^2} \llangle \EYze  \left(2 \D \frac{\nabla \phi}{\mmuz} - \left(\div \frac{\nabla \phi}{\mmuz}\right) \id \right) \EYee - \EYze \rrangle \\
 \leq C(\muz) \|\EYze\|_{\Ll^1(\supp \phi)} |\phi|_{\2} \|\EYe - \EYz\|_{\Ll^{\infty}}.
\end{multline*}
Using \eqref{EYzvsEYe} and the scaling properties of $\phi$, which is supported on a disk of radius $\lN$, we get:
\begin{equation*}
\left| \AniYe - \AniYz \right| \leq C(\bphi, \muz) \frac{\epsilon}{\omega} \lN^{-3} \|\EYz\|_{\Ll^1(\supp \phi)}  + \Ocross(\epsilon^2)
\end{equation*}
Note that the $L^1$ norm of $\EYz$ is indeed finite. We can compare it to the electric energy by writing
\begin{multline}
\label{L1toL2}
\|\EYz\|_{\Ll^1(\supp \phi)} \leq \|\EYz_{\vr}\|_{\Ll^1(\supp \phi)} + \sum_{x_i \in \supp \phi} \| \nabla \ff_{\r(x_i)} \|_{\Ll^1} 
\\
\leq  \|\EYz_{\vr}\|_{\Ll^2(\supp \phi)} \left( \lN^{2} \right)^{1/2} + \Points(\supp \phi) N^{-1/2},
\end{multline} 
up to some universal multiplicative constant. In \eqref{L1toL2} we first have re-introduced a truncation (using the nearest-neighbor distance $\r$ defined in Section \ref{sec:electricfield}) and then used Hölder's inequality. The contribution of the truncation terms $\nabla \ff_{\r(x_i)}$ can be bounded explicitly, using the fact that $\r(x_i)$ is, by definition, always smaller than $N^{-1/2}$, and the elementary bound (see \cite[Sec. 3.3.]{SerfatyCLT}):
$\int_{\R^2} | \nabla \ff_{(\eta)} | \leq C \eta.$

We can thus write:
\begin{multline*}
\left| \Ani\left(\YN, \muep, \frac{\nabla \phi}{\mmuz}\right) - \Ani\left(\YN, \muz, \frac{\nabla \phi}{\mmuz}\right) \right| \\
\leq C(\bphi, \muz) \frac{\epsilon}{\omega}  \lN^{-3}  \left(\|\EYz_{\vr}\|_{\Ll^2(\supp \phi)} \left( \lN^{2} \right)^{1/2} + \Points(\supp \phi) N^{-1/2} \right)  + \Ocross(\epsilon^2).
\end{multline*}
The $L^2$ norm of the truncated electric field corresponds to the square root of the electric energy as defined in \eqref{def:Ele}, so we write
$$
\|\EYz_{\vr}\|_{\Ll^2(\supp \phi)} \left( \lN^{2} \right)^{1/2} = \frac{\sqrt{\Ele(\supp(\phi))}}{N^{1/4}} (\lN N^{1/4}) \leq \frac{\Ele(\supp(\phi))}{N^{1/2}} + \lN^2 N^{1/2}, 
$$
where we expect both terms to be of the same size. 
Combining with \eqref{anitransporteB} and keeping only the dominant terms, we obtain:
\begin{equation*}
\left|\Ani\left(\YN, \muz, \frac{\nabla \phi}{\mmuz}\right) - \Ani\left(\XN, \muz, \frac{\nabla \phi}{\mmuz}\right)\right| 
\leq 
C(\bphi, \beta, \muz)  \frac{\epsilon}{\omega} \frac{\lN^{-3}}{N^{1/2}} \left( \EnergiePoints + N \lN^2 \right) + \Ocross(\epsilon^2)
\end{equation*}
which proves the claim.
\end{proof}

\subsection{A priori bounds}
\label{sec:aprioribounds}
\begin{lemma}[Fluctuations are bounded \corO{in $L^1,L^2$}]
\label{lem:FluctuationsBounded}
Assume $\{\xi_N\}_N$ is a sequence of test functions supported on a disk of radius $\lN$, and such that $|\xi_N|_{\kk} \leq M \lN^{-\kk}$ for $\kk \leq 3$, with a uniform constant $M >1$. Then the sequence of random variables $\left\lbrace \langle \xi_N, \LN \rangle \right\rbrace_N$ is bounded in $\Ll^1(\PNbeta)$, $\Ll^2(\PNbeta)$ by some constant (depending only on $\beta, \muz$) times $M$.
\end{lemma}
\begin{proof}
This is a consequence of \cite[Corollary 2.1]{SerfatyCLT}, which proves in fact boundedness of the \textit{exponential moments}, with an explicit bound in terms of the constant $M$. In particular, taking the parameter $\tau = \frac{1}{CM}$ in \cite[equation (2.21)]{SerfatyCLT} and applying Jensen's inequality yields a control by (some constant times) $M$ in $\Ll^1$ and $\Ll^2$ norm.
\end{proof}

\begin{lemma}[Bound on the electric energy, the number of points, the anisotropy]
\label{lem:aprioribounds}
Let $\Omega$ be a disk of radius $r$ contained in the interior of $\Sigma$ (the support of $\muz$). Then
\begin{enumerate}
\item $\Ele(\Omega)$ is bounded in $\Ll^1(\PNbeta)$, $\Ll^2(\PNbeta)$ by some constant (depending only on $\beta, \muz$)  times $N r^2$.
\item  $\Points(\Omega)$ is bounded in $\Ll^1(\PNbeta)$, $\Ll^2(\PNbeta)$ by some constant (depending only on $\beta, \muz$)  times $N r^2$.
\item If $\psi$ is $C^1$ and supported on $\Omega$, then $\Ani(\XN, \muz, \psi)$  is bounded in $\Ll^1(\PNbeta)$, $\Ll^2(\PNbeta)$ by some constant (depending only on $\beta, \muz$)  times $|\psi|_{\1} N r^2$.
\end{enumerate}
Moreover the constants are uniform for disks at a given distance of the boundary $\partial \Sigma$.
\end{lemma}
\begin{proof}
The first two points follow from the \textit{local laws} for two-dimensional Coulomb gases, obtained in \cite{LebLL, BBNYA} and refined in \cite{AS}, see also \cite[Proposition 3.7]{SerfatyCLT} for a summary. The bound on the anisotropy can be inferred from combining \cite[Proposition 4.3]{SerfatyCLT} and the aforementioned controls on both the electric energy and the number of points.
\end{proof}

\bibliographystyle{alpha}
\bibliography{LocalCLT}

\newcommand{\etalchar}[1]{$^{#1}$}
\begin{thebibliography}{AHM{\etalchar{+}}15}

\bibitem[AHM{\etalchar{+}}15]{ameur2015random}
Yacin Ameur, Haakan Hedenmalm, Nikolai Makarov, et~al.
\newblock Random normal matrices and {Ward} identities.
\newblock {\em The Annals of Probability}, 43(3):1157--1201, 2015.

\bibitem[AS19]{AS}
S.~Armstrong and S.~Serfaty.
\newblock Local laws and rigidity for {C}oulomb gases at any temperature.
\newblock {\em arXiv preprint arXiv:1906.09848}, 2019.

\bibitem[BBNY17]{BBNYA}
R.~Bauerschmidt, P.~Bourgade, M.~Nikula, and H.-T. Yau.
\newblock Local density for two-dimensional one-component plasma.
\newblock {\em Communications in Mathematical Physics}, 356(1):189--230, 2017.

\bibitem[BBNY19]{BBNY}
R.~Bauerschmidt, P.~Bourgade, M.~Nikula, and H.-T. Yau.
\newblock The two-dimensional {Coulomb} plasma: quasi-free approximation and
  central limit theorem.
\newblock {\em Advances in Theoretical and Mathematical Physics}, 23(4), 2019.

\bibitem[BEYY16]{bourgade2016fixed}
Paul Bourgade, Laszlo Erd{\H{o}}s, Horng-Tzer Yau, and Jun Yin.
\newblock Fixed energy universality for generalized {Wigner} matrices.
\newblock {\em Communications on Pure and Applied Mathematics},
  69(10):1815--1881, 2016.

\bibitem[BFG15]{bekerman2015transport}
Florent Bekerman, Alessio Figalli, and Alice Guionnet.
\newblock Transport maps for $\beta$-matrix models and universality.
\newblock {\em Communications in mathematical physics}, 338(2):589--619, 2015.

\bibitem[BG13]{BG1}
G.~Borot and A.~Guionnet.
\newblock Asymptotic expansion of $\beta$ matrix models in the one-cut regime.
\newblock {\em Communications in Mathematical Physics}, 317(2):447--483, 2013.

\bibitem[BLS18]{BLS}
F.~Bekerman, T.~Lebl{\'e}, and S~Serfaty.
\newblock {CLT} for fluctuations of $\beta$-ensembles with general potential.
\newblock {\em Electronic Journal of Probability}, 23, 2018.

\bibitem[JL20]{johansson2020multivariate}
Kurt Johansson and Gaultier Lambert.
\newblock Multivariate normal approximation for traces of random unitary
  matrices.
\newblock {\em arXiv preprint arXiv:2002.01879}, 2020.

\bibitem[Joh97]{johansson1997random}
Kurt Johansson.
\newblock On random matrices from the compact classical groups.
\newblock {\em Annals of mathematics}, pages 519--545, 1997.

\bibitem[Leb17]{LebLL}
T.~Lebl{\'e}.
\newblock Local microscopic behavior for 2{D} {C}oulomb gases.
\newblock {\em Probability Theory and Related Fields}, 169(3-4):931--976, 2017.

\bibitem[LS18]{LebSerCLT}
T.~Lebl\'{e} and S.~Serfaty.
\newblock Fluctuations of two dimensional {C}oulomb gases.
\newblock {\em Geom. Funct. Anal.}, 28(2):443--508, 2018.

\bibitem[Pas06]{Pasturcut}
L.~Pastur.
\newblock Limiting laws of linear eigenvalue statistics for {H}ermitian matrix
  models.
\newblock {\em J. Math. Phys.}, 47(10):103303, 22, 2006.

\bibitem[Ser20]{SerfatyCLT}
S.~Serfaty.
\newblock Gaussian fluctuations and free energy expansion for {2D} and {3D}
  {C}oulomb gases at any temperature.
\newblock {\em arXiv preprint arXiv:2003.11704}, 2020.

\bibitem[Shc14]{Shcherbina_2014}
M.~Shcherbina.
\newblock Change of variables as a method to study general $\beta$-models: Bulk
  universality.
\newblock {\em Journal of Mathematical Physics}, 55(4):043504, apr 2014.

\bibitem[SS18]{serfaty2018quantitative}
S.~Serfaty and J.~Serra.
\newblock Quantitative stability of the free boundary in the obstacle problem.
\newblock {\em Analysis \& PDE}, 11(7):1803--1839, 2018.

\bibitem[ZW06]{zabrodin2006large}
A~Zabrodin and P~Wiegmann.
\newblock {Large-N expansion for the 2D Dyson gas}.
\newblock {\em Journal of Physics A: Mathematical and General}, 39(28):8933,
  2006.

\end{thebibliography}

\end{document}